\documentclass{amsart}
\usepackage{amsfonts, amsbsy, amsmath, amssymb}
\usepackage{tikz}

\newtheorem{thm}{Theorem}[section]
\newtheorem{lem}[thm]{Lemma}
\newtheorem{cor}[thm]{Corollary}
\newtheorem{prop}[thm]{Proposition}

\newtheorem{conj}[thm]{Conjecture}

\newtheorem{rmk}[thm]{Remark}

\newtheorem{ques}[thm]{Question}
\numberwithin{equation}{section}

\theoremstyle{definition}
\newtheorem{defn}[thm]{Definition}

\allowdisplaybreaks

\newcommand{\f}{\Bbb F}
\newcommand{\ch}{\text{\rm char}\,}

\newcommand{\fin}{f_{\text{\rm inv}}}
\newcommand{\tr}{\text{\rm Tr}}
\newcommand{\rank}{\text{\rm rank\,}}
\newcommand{\res}{\text{\rm Res}}
\newcommand{\wt}{\text{\rm wt}}

\begin{document}

\title[Sum-Freedom of Binary and $q$-ary Functions]{Further Results on Sum-Freedom of Binary and $q$-ary Functions}

\author[Xiang-dong Hou]{Xiang-dong Hou}
\address{Department of Mathematics and Statistics,
University of South Florida, Tampa, FL 33620}
\email{xhou@usf.edu}

\author[S. Zhao]{Shujun Zhao}
\address{Department of Mathematics and Statistics,
University of South Florida, Tampa, FL 33620}
\email{shujunz@usf.edu}



\begin{abstract}
The notion of sum-freedom of binary functions was introduced recently by C. Carlet as a generalization of the APN functions used in cryptography; the $q$-ary version of the notion is a natural extension. For each integer $k$ with $0\le k\le n$, there is a $k$th order sum-free function on $\Bbb F_{2^n}$. It is also known that when $k/n$ is not close to 0 or 1, the multiplicative inverse function on $\Bbb F_{2^n}$ is not $k$th order sum-free. We generalize these two results to $q$-ary functions. APN functions have a coding theoretic characterization. We generalize the characterization to sum-free functions of arbitrary order over any finite field. It is well known that the Welch functions is 2nd order sum-free. We give an alternative proof for this result which leads to a more general algebraic question. We also investigate that the 3rd order sum-freedom of the Welch function and power functions of algebraic degree 3. We formulate a conjecture about the 3rd order sum-freedom of the Welch function which is supported by strong numerical evidence. 
\end{abstract}

\maketitle

\section{Introduction}

Let $\f_q$ denote the finite field with $q$ elements. A function $f$ from $\f_{q^n}$ (or $\f_q^n$) to itself is said to be {\em $k$th order sum-free} if for every $k$-dimensional $\f_q$-affine subspace $A$ of $\f_{q^n}$,
\begin{equation}\label{1.1}
\sum_{x\in A}f(x)\ne 0.
\end{equation}
The binary version of sum-free functions were introduced by Carlet \cite{Carlet-DCC-2024, Carlet-JC-2025}, and the $q$-ary version of sum-freedom was studied in \cite{Ebeling-Hou-Rydell-Zhao-FFA-2016}. Binary 2nd order sum-free functions are precisely {\em almost perfect nonlinear} (APN) functions, which have been extensively studied for their applications in cryptography \cite{Carlet-VBF-2010, Dobbertin-IEEE-IT-1999, Dobbertin-IC-1999, Dobbertin-2001, Hollmann-Xiang-FFA-2001, Nyberg-LNCS-547-1991, Nyberg-LNCS-1994}.

A natural question is this: given $k$, does there exist a $k$th order sum-free function on $\f_{q^n}$? For $q=2$, the answer is positive.

\begin{thm}\label{T1.1} \cite[Proposition~1]{Carlet-JC-2025} For $0\le k\le n$, the function $f(X)=X^{2^k-1}$ defined on $\f_{2^n}$ is $k$th order sum-free.
\end{thm}

Let $\fin:\f_{2^n}\to\f_{2^n}$ be the multiplicative inverse function defined as $\fin(x)=x^{-1}$ for $x\in\f_{2^n}^*$ and $\fin(0)=0$. It is well known that $\fin$ is 2nd order sum-free, equivalently, $(n-2)$-th order sum-free, if and only if $n$ is odd \cite{Carlet-DCC-2024, Nyberg-LNCS-1994}. It was conjectured by Carlet that $\fin$ is not $k$th order sum-free for $3\le k\le n-3$ \cite{Carlet-DCC-2024}. This conjecture has been a driving force for the recent works on sum-free functions. Carlet's conjecture has been confirmed when $n$ is not a prime \cite{Hou-Zhao-DCC-2026}. The conjecture is also known to be true when $k/n$ is not close to 0 or 1 \cite{Carlet-Hou-DCC-2025, Hou-Zhao-FFA-2026}.

\begin{thm}\label{T1.2}
\cite[Theorem~ 5.1]{Hou-Zhao-FFA-2026} Assume $3\le k\le n-3$. If
\[
n\ge \frac{13}3k-1.37\quad\text{or}\quad n\le\frac{13}{10}k+0.41,
\]
$\fin$ is not $k$th order sum-free on $\f_{2^n}$.
\end{thm}

In the present paper, we will generalize Theorems~\ref{T1.1} and \ref{T1.2} to the $q$-ary case; see Theorems~\ref{T2.3} and \ref{T3.3}, respectively. Generalizations of binary results to the $q$-ary setting are not always straightforward. In the two generalizations considered here, although the proofs follow the same ideas as in the binary case, additional techniques are needed. For Theorem~\ref{T2.3}, we first prove a useful identity which does not seem to be well known. Moreover, Theorem~\ref{T2.3} is further extended to give a family of $k$th order sum-free functions. When proving Theorem~\ref{T3.3} using the Lang-Weil bound, we show by induction that a certain polynomial has an absolutely irreducible factor.

APN functions can be characterized in terms of a subcode of a Reed--Muller code \cite{Carlet-2021, Carlet-Charpin-Zinoviev-DCC-1998}. We will generalize this characterization to sum-free functions of any order over any finite field (Theorem~\ref{t4.6}). The proof of this result critically depends a theorem by Delsarte, Goethals and MacWilliams that determines the minimum weight codewords of the Reed-Muller code. Shortly after the previous version of the present paper was submitted to arXiv, we learned that the binary case of our result had been covered in \cite{Heering-Kaspers-Taranchuk-2026}. However, to our best knowledge, the $q$-ary case is new.

The Welch function on $\f_{2^n}$, where $n=2m+1$, is defined by $W_n(X)=X^{2^m+3}$. It is known that the Welch function is APN, i.e., 2nd order sum-free \cite{Dobbertin-IEEE-IT-1999}. In fact, it is among a handful of well-known power APN functions \cite{Dobbertin-2001}. Section~\ref{sec5} of the present paper is devoted to a study of sum-freedom of the Welch function. We give an alternative proof for the 2nd sum-freedom of the Welch function. This proof leads to a general algebraic question and our proof can be viewed as an answer to that question in a very special case. It is not known whether $W_n(X)$ is 3rd order sum-free for $m\ge 3$. We provide a necessary and sufficient condition for $W_n(X)$ not to be 3rd order sum-free in terms of a subdeterminant of a Dickson matrix. Based on computer experiments, we conjecture that $W_n(X)$ is not 3rd order sum-free for $m\ge 3$. We also extend our discussion from the Welch function to power functions of algebraic degree 3.

\section{A Generalization of Theorem~\ref{T1.1}}\label{sec2}

\subsection{A $q$-ary version of Theorem~\ref{T1.1}}\

For a nonnegative integer $m$ with base $q$ representation
\[
m=m_0q^0+m_1q^1+\cdots,\quad 0\le m_i\le q-1,
\]
its base $q$ weight is defined as $\wt_q(m)=m_0+m_1+\cdots$.

\begin{lem}\label{L2.1}
Let $m$ be a positive integer such that $\wt_q(m)\le k(q-1)$. Then
\[
\sum_{a_1,\dots,a_k\in\f_q}(Y+a_1X_1+\cdots+a_kX_k)^m=\sum_{a_1,\dots,a_k\in\f_q}(a_1X_1+\cdots+a_kX_k)^m.
\]
\end{lem}

\begin{proof}
We have
\begin{align*}
&\sum_{a_1,\dots,a_k\in\f_q}(Y+a_1X_1+\cdots+a_kX_k)^m\cr
=\,&\sum_{a_1,\dots,a_k\in\f_q}\sum_{i_0,i_1,\dots,i_k}\binom m{i_0,i_1,\dots,i_k}Y^{i_0}(a_1X_1)^{i_1}\cdots(a_kX_k)^{i_k}\cr
=\,&\sum_{i_0,i_1,\dots,i_k}\binom m{i_0,i_1,\dots,i_k}Y^{i_0}X_1^{i_1}\cdots X_k^{i_k}\Bigl(\sum_{a_1\in\f_q}a_1^{i_1}\Bigr)\cdots\Bigl(\sum_{a_k\in\f_q}a_k^{i_k}\Bigr)\cr
=\,&(-1)^k\sum_{\substack{i_0,i_1,\dots,i_k\cr 0<i_j\equiv0\,(\text{mod}\, q-1),\,1\le j\le k}}\binom m{i_0,i_1,\dots,i_k}Y^{i_0}X_1^{i_1}\cdots X_k^{i_k}.
\end{align*}
In the above, $\binom m{i_0,i_1,\dots,i_k}\ne 0$ if and only if the sum $i_0+i_1+\cdots+i_k$ has no carries in base $p$ ($=\ch \f_q$), hence only if $i_0+i_1+\cdots+i_k$ has no carries in base $q$. When $i_0+i_1+\cdots+i_k$ ($=m$) has no carries in base $q$, where $0<i_j\equiv 0\pmod{q-1}$ for $1\le j\le k$ and $\wt_q(m)\le k(q-1)$, we have $i_1+\cdots+i_k=m$, whence $i_0=0$. Therefore,
\[
\sum_{a_1,\dots,a_k\in\f_q}(Y+a_1X_1+\cdots+a_kX_k)^m=\sum_{a_1,\dots,a_k\in\f_q}(a_1X_1+\cdots+a_kX_k)^m.
\]
\end{proof}

\begin{lem}\label{L2.2}
Let $k$ be a positive integer. Then
\[
\sum_{a_1,\dots,a_k\in\f_q}(a_1X_1+\cdots+a_kX_k)^{q^k-1}=\prod_{0\ne(a_1,\dots,a_k)\in\f_q^k}(a_1X_1+\cdots+a_kX_k).
\]
\end{lem}

\begin{proof}
Let
\[
F(X_1,\dots,X_k)=\sum_{a_1,\dots,a_k\in\f_q}(a_1X_1+\cdots+a_kX_k)^{q^k-1},
\]
and treat it as a polynomial in $X_1$ over $\f_q[X_2,\dots,X_k]$. Let $\partial^iF$ denote the $i$th Hasse derivative of $F$ with respect to $X_1$ \cite[p.28]{Goldschmidt-2003}. For $0\le i<q-1$,
\[
\partial^iF=\binom{q^k-1}i\sum_{a_1,\dots,a_k\in\f_q}(a_1X_1+\cdots+a_kX_k)^{q^k-1-i}a_1^i.
\]
For $(b_2,\dots,b_k)\in\f_q^{k-1}$, we have
\begin{align*}
&\partial^iF\Bigm|_{X_1=-(b_2X_2+\cdots+b_kX_k)}\cr
=\,&\binom{q^k-1}i\sum_{a_1,\dots,a_k\in\f_q}\bigl((a_2-a_1b_2)X_2+\cdots+(a_k-a_1b_k)X_k\bigr)^{q^k-1-i}a_1^i\cr
=\,&\binom{q^k-1}i\sum_{a_2,\dots,a_k\in\f_q}(a_2X_2+\cdots+a_kX_k)^{q^k-1-i}\sum_{a_1\in\f_q}a_1^i\cr
=\,&0.
\end{align*}
Thus $-(b_2X_2+\cdots+b_kX_k)$ is a root of $F$ with multiplicity $\ge q-1$ \cite[Corollary~2.5.14]{Goldschmidt-2003}. Therefore
\begin{equation}\label{eq:|F}
\prod_{(a_1,\dots,a_k)\sim(1,b_2,\dots,b_k)}(a_1X_1+\cdots+a_kX_k)\mid F(X_1,\dots,X_k),
\end{equation}
where $(a_1,\dots,a_k)\sim(1,b_2,\dots,b_k)$ means that $(a_1,\dots,a_k)=u(1,b_2,\dots,b_k)$ for some $u\in\f_q^*$. Since $F(X_1,\dots,X_k)$ is symmetric in $X_1,\dots,X_k$, \eqref{eq:|F} implies that
\[
\prod_{0\ne(a_1,\dots,a_k)\in\f_q^k}(a_1X_1+\cdots+a_kX_k)\mid F(X_1,\dots,X_k).
\]
Comparing the degrees on both sides gives
\begin{equation}\label{eq:c}
F(X_1,\dots,X_k)=c\prod_{0\ne(a_1,\dots,a_k)\in\f_q^k}(a_1X_1+\cdots+a_kX_k)
\end{equation}
for some $c\in\f_q$. The coefficient of $X_1^{(q-1)q^{k-1}}X_2^{(q-1)q^{k-2}}\cdots X_k^{(q-1)q^0}$ on the RHS of \eqref{eq:c} is
\[
c\Bigl(\prod_{a_1\in\f_q^*}a_1\Bigr)^{q^{k-1}}\cdots\Bigl(\prod_{a_k\in\f_q^*}a_k\Bigr)^{q^0}=c(-1)^k.
\]
The coefficient of $X_1^{(q-1)q^{k-1}}X_2^{(q-1)q^{k-2}}\cdots X_k^{(q-1)q^0}$ in $F(X_1,\dots,X_k)$ is 
\begin{align*}
&\sum_{a_1,\dots,a_k\in\f_q}\binom{q^k-1}{(q-1)q^{k-1},\dots,(q-1)q^0}a_1^{(q-1)q^{k-1}}\cdots a_k^{(q-1)q^0}\cr
=\,&\Bigl(\sum_{a_1\in\f_q}a_1^{q-1}\Bigr)^{q^{k-1}}\cdots\Bigl(\sum_{a_1\in\f_q}a_1^{q-1}\Bigr)^{q^0}=(-1)^k.
\end{align*}
Hence $c=1$.
\end{proof} 

\begin{thm}\label{T2.3} ($q$-ary version of Theorem~\ref{T1.1}) For $0\le k\le n$, $f(X)=X^{q^k-1}$ is $k$th order sum-free on $\f_{q^n}$.
\end{thm}

\begin{proof}
Let $A$ be a $k$-dimensional $\f_q$-affine subspace of $\f_{q^n}$, and write $A=y+\langle u_1,\dots,u_k\rangle$, where $y\in\f_{q^n}$ and $u_1,\dots,u_k\in\f_{q^n}$ are linearly independent over $\f_q$. By Lemmas~\ref{L2.1} and \ref{L2.2},
\begin{align*}
\sum_{x\in A}f(x)\,&=\sum_{a_1,\dots,a_k\in\f_q}(y+a_1u_1+\cdots+a_ku_k)^{q^k-1}\cr
&=\prod_{0\ne(a_1,\dots,a_k)\in\f_q^n}(a_1u_1+\cdots+a_ku_k)\cr
&\ne 0.
\end{align*}
\end{proof}

\begin{rmk}\label{R2.4}\rm
If $\wt_q(m)<k(q-1)$, it follows from the proof of Lemma~\ref{L2.1} that 
\[
\sum_{a_1,\dots,a_k\in\f_q}(Y+a_1X_1+\cdots+a_kX_k)^m=0.
\]
Hence every function $f:\f_{q^n}\to\f_{q^n}$ with algebraic degree $<k(q-1)$ is not $k$th order sum-free. (The algebraic degree of $f$ is the largest base $q$ weight of the degrees of the monomials in $f$.)
\end{rmk}

\subsection{An extension of Theorem~\ref{T2.3}}\

First, we prove an extension of Lemma~\ref{L2.2}.

\begin{prop}\label{PA3}
Let $q$ be a prime power and $k, l$ be positive integers. Then
\begin{align}\label{eq:5}
&\Bigl(\sum_{a_1,\dots,a_k\in\f_q}(a_1X_1+\cdots+a_kX_k)^{(q-1)(1+q^l+q^{2l}+\cdots+q^{(k-1)l})} \Bigr)^{(q^l-1)/(q-1)}\\
&=\prod_{0\ne(b_1,\dots,b_k)\in\f_{q^l}^k}(b_1X_1+\cdots+b_kX_k).\nonumber
\end{align}
\end{prop}

\begin{proof}
Let
\[
F(X_1,\dots,X_k)=\sum_{a_1,\dots,a_k\in\f_q}(a_1X_1+\cdots+a_kX_k)^{(q-1)(1+q^l+q^{2l}+\cdots+q^{(k-1)l})},
\]
and treat it as a polynomial in $X_1$ over $\f_q[X_2,\dots,X_k]$. Let $\partial^iF$ be the $i$th Hasse derivative of $F$ with respect to $X_1$. For $0\le i<q-1$,
\[
\partial^iF=\binom{q-1}i \sum_{a_1,\dots,a_k\in\f_q}(a_1X_1+\cdots+a_kX_k)^{(q-1)(1+q^l+q^{2l}+\cdots+q^{(k-1)l})-i}a_1^i.
\]
For $(b_2,\dots,b_k)\in\f_{q^l}^{k-1}$, we have
\begin{align}\label{eq:6}
&\partial^iF\big|_{X_1=-(b_2X_2+\cdots+b_kX_k)}\\
=\,&\binom{q-1}i\sum_{a_1,\dots,a_k\in\f_q}\bigl((a_2-a_1b_2)X_2+\cdots+(a_k-a_1b_k)X_k\bigr)^{(q-1)(1+q^l+\cdots+q^{(k-1)l})-i}a_1^i\cr
=\,&\binom{q-1}i\sum_{a_1,\dots,a_k\in\f_q}\, a_1^i \bigl((a_2-a_1b_2)X_2+\cdots+(a_k-a_1b_k)X_k\bigr)^{q-1-i}\cr
&\kern8em\cdot \prod_{j=1}^{k-1}\bigl((a_2-a_1b_2)X_2^{q^{jl}}+\cdots+(a_k-a_1b_k)X_k^{q^{jl}}\bigr)^{q-1}.\nonumber
\end{align}
In the above,
\begin{align}\label{eq:7}
&a_1^i \bigl((a_2-a_1b_2)X_2+\cdots+(a_k-a_1b_k)X_k\bigr)^{q-1-i}\\
&\cdot\prod_{j=1}^{k-1}\bigl((a_2-a_1b_2)X_2^{q^{jl}}+\cdots+(a_k-a_1b_k)X_k^{q^{jl}}\bigr)^{q-1}\cr
=\,&\sum_{s_2+\cdots+s_k=k(q-1)-i}f_{s_2,\dots,s_k}(a_2-a_1b_2)^{s_2}\cdots(a_k-a_1b_k)^{s_k}a_1^i,\nonumber
\end{align}
where $f_{s_2,\dots,s_k}\in\f_q[X_2,\dots,X_k]$. Since $s_2+\cdots+s_k=k(q-1)-i>(k-1)(q-1)$, there exists $2\le i\le k$ such that $s_i\ge q$. For this $i$,
\[
(a_i-a_1b_i)^{s_i}=(a_i-a_1b_i)^q(a_i-a_1b_i)^{s_i-q}=(a_i-a_1b_i^q)(a_i-a_1b_i)^{s_i-q},
\]
where the total degree in $a_1,a_i$ on the RHS is $s_i-q+1$. Therefore, the product in \eqref{eq:7} can be expressed as a polynomial in $a_1,\dots,a_k$ with total degree $\le s_2+\cdots+s_k-q+1+i=(k-1)(q-1)<k(q-1)$. Then by \eqref{eq:6}, 
\[
\partial^iF\big|_{X_1=-(b_2X_2+\cdots+b_kX_k)}=0.
\]
Hence $(X_1+b_2X_2+\cdots+b_kX_k)^{q-1}\mid F(X_1,\dots,X_k)$. Since both sides of \eqref{eq:5} are symmetric in $X_1,\dots,X_k$, it follows that the RHS of \eqref{eq:5} divides the LHS. Comparing the total degrees (in $X_1,\dots,X_k$) on the two sides of \eqref{eq:5} gives 
\begin{align}\label{eq:8}
&\Bigl(\sum_{a_1,\dots,a_k\in\f_q}(a_1X_1+\cdots+a_kX_k)^{(q-1)(1+q^l+q^{2l}+\cdots+q^{(k-1)l})} \Bigr)^{(q^l-1)/(q-1)}\\
&=c\prod_{0\ne(b_1,\dots,b_k)\in\f_{q^l}^k}(b_1X_1+\cdots+b_kX_k)\nonumber
\end{align}
for some $c\in\f_{q^l}^*$.

The coefficient of $X_1^{(q^l-1)q^{(k-1)l}}X_2^{(q^l-1)q^{(k-2)l}}\cdots X_k^{(q^l-1)q^{0l}}$ on the RHS of \eqref{eq:8} is
\[
c\Bigl(\prod_{b_1\in\f_{q^l}^*}b_1\Bigr)^{q^{(k-1)l}}\cdots\Bigl(\prod_{b_b\in\f_{q^l}^*}b_k\Bigr)^{q^{0l}}=c(-1)^k.
\]
The coefficient of the same term on the LHS of \eqref{eq:8} is 
\begin{align}
&\sum_{a_1,\dots,a_k\in\f_q}a_1^{(q^l-1)q^{(k-1)l}}\cdots a_k^{(q^l-1)q^{0l}}\cr
=\,&\Bigl(\sum_{a_1\in\f_q}a_1^{q^l-1}\Bigr)^{q^{(k-1)l}}\cdots\Bigl(\sum_{a_k\in\f_q}a_k^{q^l-1}\Bigr)^{q^{0l}}=(-1)^k.
\end{align}
Hence $c=1$.
\end{proof}

We believe that the following fact is well known, but we couldn't locate a proper reference.

\begin{lem}\label{LA}
Let $m,n$ be positive integers such that $\gcd(m,n)=1$. If $u_1,\dots,u_k\in\f_{q^m}$ are linearly independent over $\f_q$, then they are also linearly independent over $\f_{q^n}$; that is, $\f_{q^m}$ and $\f_{q^n}$ are linearly disjoint over $\f_q$.
\end{lem}

\begin{proof}
Let $u_1,\dots,u_k\in\f_{q^m}$. By \cite[Lemma~2.30]{Hou-ams-gsm-2018}, $u_1,\dots,u_k$ are linearly independent over $\f_q$ if and only if
\begin{equation}\label{eq:1}
\rank\left[\begin{matrix}
u_1&\cdots&u_k\cr
u_1^q&\cdots&u_k^q\cr
\vdots&&\vdots\cr
u_1^{q^{m-1}}&\cdots&u_k^{q^{m-1}}
\end{matrix}
\right]=k.
\end{equation}
In the same way, $u_1,\dots,u_k$ are linearly independent over $\f_{q^n}$ if and only if
\begin{equation}\label{eq:2}
\rank\left[\begin{matrix}
u_1&\cdots&u_k\cr
u_1^{q^n}&\cdots&u_k^{q^n}\cr
\vdots&&\vdots\cr
u_1^{q^{n(m-1)}}&\cdots&u_k^{q^{n(m-1)}}
\end{matrix}
\right]=k.
\end{equation}
Since $\gcd(m,n)=1$, the matrix in \eqref{eq:2} is a row permutation of the matrix in \eqref{eq:1}. Hence \eqref{eq:1} and \eqref{eq:2} are equivalent.
\end{proof}

\begin{rmk}\label{R2.7}\rm
A more conceptual proof of Lemma~\ref{LA} is as follows: Put $\f_{q^m}$ and $\f_{q^n}$ in a fixed algebraic closure $\overline\f_q$. The $\f_q$-map $f:\f_{q^m}\otimes_{\f_q}\f_{q^n}\to \f_{q^m}[\f_{q^n}]=\f_{q^{mn}}$ sending $a\otimes b$ to $ab$ is onto. Since $\dim_{\f_q}\f_{q^m}\otimes_{\f_q}\f_{q^n}=mn=\dim_{\f_q}\f_{q^{mn}}$, $f$ is an $\f_q$-isomorphism. By \cite[Exercise~VI.2.1]{Hungerford-1980}, $\f_{q^m}$ and $\f_{q^n}$ are linearly disjoint over $\f_q$.
\end{rmk}

\begin{cor}\label{CA2}
Let $n,k,l$ be positive integers such that $2\le k\le n$. Then 
\[
X^{(q-1)(1+q^l+q^{2l}+\cdots+q^{(k-1)l})}
\]
is $k$th order sum-free if and only if $\gcd(l,n)=1$.
\end{cor}

\begin{proof}
($\Leftarrow$) Let $A=y+\langle u_1,\dots,u_k\rangle$ be a $k$-dimensional $\f_q$-affine subspace of $\f_{q^n}$, where $u_1,\dots,u_k\in\f_{q^n}$ are linearly independent over $\f_q$ and $y\in\f_{q^n}$. By Lemma~\ref{LA}, $u_1,\dots,u_k$ are also linearly independent over $\f_{q^l}$. Then by Lemma~\ref{L2.1} and Proposition~\ref{PA3},
\[
\sum_{x\in A}x^{(q-1)(1+q^l+\cdots+q^{(k-1)l})}=\sum_{a_1,\dots,a_k\in\f_q}(a_1u_1+\cdots+a_ku_k)^{(q-1)(1+q^l+\cdots+q^{(k-1)l})}\ne0.
\]

\medskip
($\Rightarrow$) Assume to the contrary that $d=\gcd(l,n)>1$. Then there exist $u_1,\dots,u_k\in\f_{q^n}$ which are linearly independent over $\f_q$ but linearly dependent over $\f_{q^d}$. Then by Proposition~\ref{PA3},
\[
\sum_{a_1,\dots,a_k\in\f_q}(a_1u_1+\cdots+a_ku_k)^{(q-1)(1+q^l+\cdots+q^{(k-1)l})}=0,
\]
which is a contradiction.
\end{proof}

\begin{rmk}\label{R2.9}\rm
The binary form of Corollary~\ref{CA2} was stated in \cite[\S5.3]{Carlet-ePrint-2024} and a proof was outlined there. The proof of Corollary~\ref{CA2} given above is different.
\end{rmk}

If two integers belong to the same $q$-cyclotomic coset modulo $q^n-1$, i.e, $k_1\equiv q^ik_2\pmod{q^n-1}$ for some $i\ge 0$, then 
\[
X^{k_1}\equiv(X^{k_2})^{q^i}\pmod{X^{q^n}-X},
\]
and hence the power functions $X^{k_1}$ and $X^{k_2}$ on $\f_{q^n}$ have the same orders of sum-freedom. This fact, together with Corollary~\ref{CA2}, gives a family of $k$th order sum-free functions. Two multi subsets $I_1$ and $I_2$ of $\Bbb Z_n$ are said to be {\em equivalent}, denoted as $I_1\overset n\sim I_2$, if there is an invertible affine polynomial $\alpha(X)\in\Bbb Z[X]$, i.e., $\alpha(X)=aX+b$, where $a\in\Bbb Z_n^\times$ and $b\in\Bbb Z_n$, such that $\alpha(I_1)=I_2$. The family of $k$th order sum-free functions on $\f_{q^n}$ is
\[
\bigl\{X^{(q-1)(q^{i_1}+\cdots+q^{i_k})}:\{i_1,\dots,i_k\}\overset n\sim\{0,1,\dots,k-1\}\bigr\}.
\]
For example, when $q=2$, $k=3$ and $n=5$,
\[
\{0,1,3\}=\alpha(\{0,1,2\}),\quad \{0,2,3\}=\beta(\{0,1,2\}),
\]
where $\alpha(X)=3X$ and $\beta(X)=2X+3$ are invertible affine polynomials in $\Bbb Z_5[X]$. Therefore, both $X^{2^0+2^1+2^3}=X^{11}$ and $X^{2^0+2^2+2^3}=X^{13}$ are 3rd order sum-free on $\f_{2^5}$. The 3rd order sum-freedom of $X^{13}$ on $\f_{2^5}$ was first observed in \cite[\S3.1]{Carlet-JC-2025}.


\section{A $q$-ary Version of Theorem~\ref{T1.2}}

The natural generalization of the binary multiplicative inverse function $\fin$ is not the $q$-ary multiplicative inverse function, but rather the function $g_{q-1}:\f_{q^n}\to\f_{q^n}$ defined by 
\[
g_{q-1}(x)=\begin{cases}
1/x^{q-1}&\text{if}\ x\ne 0,\cr
0&\text{if}\ x=0.
\end{cases}
\]
See \cite[\S4.4]{Ebeling-Hou-Rydell-Zhao-FFA-2016} for the justification for this generalization.

Many properties of the function $\fin$ are also possessed by the $q$-ary version $g_{q-1}$. For example, for $n\ge 2$, $g_{q-1}$ on $\f_{q^n}$ is 2nd order sum-free if and only if $n$ is odd \cite[Proposition~4.7]{Ebeling-Hou-Rydell-Zhao-FFA-2016}; $g_{q-1}$ is $k$th order sum-free if and only if it is $(n-k)$-th order sum-free \cite[Theorem~4.10]{Ebeling-Hou-Rydell-Zhao-FFA-2016}. However, these two functions do not always behave the same. For example, for $q=3,5$ and $n=7$, $g_{q-1}$ is $k$th order sum-free for all $1\le k\le 6$. Therefore, if we are to formulate a conjecture about $g_{q-1}$ emulating Carlet's conjecture on $\fin$, we will need to assume $n\ge 8$.

In this section, we will prove a $q$-ary version of Theorem~\ref{T1.2}. First, we need to recall some background from \cite[\S4.4]{Ebeling-Hou-Rydell-Zhao-FFA-2016}. For $k>0$, define
\begin{align*}
\Delta(X_1,\dots,X_k)\,&=\left|\begin{matrix}
X_1&\cdots& X_k\cr
X_1^q&\cdots& X_k^q\cr
\vdots&&\vdots\cr
X_1^{q^{k-1}}&\cdots& X_k^{q^{k-1}}
\end{matrix}
\right|\cr
&=\prod_{i=1}^k\,\prod_{a_1,\dots,a_{i-1}\in\f_q}\Bigl(X_i-\sum_{j=1}^{i-1}a_jX_j\Bigr)\in\f_q[X_1,\dots,X_k],
\end{align*}
and for $1\le i\le k$, define
\[
\Delta_i(X_1,\dots,X_k)=\left|\begin{matrix}
X_1&\cdots& X_k\cr
\vdots&&\vdots\cr
X_1^{q^{i-1}}&\cdots& X_k^{q^{i-1}}\cr
X_1^{q^{i+1}}&\cdots& X_k^{q^{i+1}}\cr
\vdots&&\vdots\cr
X_1^{q^k}&\cdots& X_k^{q^k}
\end{matrix}
\right|\in\f_q[X_1,\dots,X_k].
\]
Then $\Delta\mid \Delta_i$. The function $g_{q-1}$ is not $k$th order sum-free on $\f_{q^n}$ if and only if there exist $v_1,\dots,v_k\in\f_{q^n}$ such that $\Delta(v_1,\dots,v_k)\ne 0$ but $\Delta_1(v_1,\dots,v_k)=0$. Let
\[
F_k(X_1,\dots,X_k)=\frac{\Delta_1(X_1,\dots,X_k)}{\Delta(X_1,\dots,X_k)}\in\f_q[X_1,\dots,X_k].
\]

\begin{lem}\label{L3.1}
$F_3(X_1,X_2,X_3)$ is absolutely irreducible, i.e., irreducible in $\overline\f_q[X_1,X_2,X_3]$.
\end{lem}

\begin{proof}
Since $F_3(X_1,X_2,X_3)$ is homogeneous, it suffices to show that $F_3(X_1,X_2,1)$ is absolutely irreducible. We have
\begin{align*}
&\Delta(X_1,X_2,1)=\left|\begin{matrix}
X_1&X_2&1\cr
X_1^q&X_2^q&1\cr
X_1^{q^2}&X_2^{q^2}&1
\end{matrix}\right|\cr
&=(X_1X_2^q-X_1^qX_2)-(X_1X_2^{q^2}-X_1^{q^2}X_2)+(X_1^qX_2^{q^2}-X_1^{q^2}X_2^q)\cr
&=X_1X_2\bigl[(X_2^{q-1}-X_1^{q-1})-(X_2^{q^2-1}-X_1^{q^2-1})+(X_1^{q-1}X_2^{q^2-1}-X_1^{q^2-1}X_2^{q-1})\bigr],
\end{align*}

\begin{align*}
&\Delta_1(X_1,X_2,1)=\left|\begin{matrix}
X_1&X_2&1\cr
X_1^{q^2}&X_2^{q^2}&1\cr
X_1^{q^3}&X_2^{q^3}&1
\end{matrix}\right|\cr
&=(X_1X_2^{q^2}-X_1^{q^2}X_2)-(X_1X_2^{q^3}-X_1^{q^3}X_2)+(X_1^{q^2}X_2^{q^3}-X_1^{q^3}X_2^{q^2})\cr
&=X_1X_2\bigl[(X_2^{q^2-1}-X_1^{q^2-1})-(X_2^{q^3-1}-X_1^{q^3-1})+(X_1^{q^2-1}X_2^{q^3-1}-X_1^{q^3-1}X_2^{q^2-1})\bigr].
\end{align*}
Write
\[
F_3(X_1,X_2,1)=C_{q^2-q}+C_{q^2-q+1}+\cdots+C_{q^3-q},
\]
where $C_i\in\f_q[X_1,X_2]$ is homogeneous of degree $i$. Then 
\begin{align}\label{eq:C+...}
&(C_{q^2-q}+\cdots+C_{q^3-q})\Bigl(1-\frac{X_2^{q^2-1}-X_1^{q^2-1}}{X_2^{q-1}-X_1^{q-1}}+\frac{X_1^{q-1}X_2^{q^2-1}-X_1^{q^2-1}X_2^{q-1}}{X_2^{q-1}-X_1^{q-1}}\Bigr)\\
&=\frac{X_2^{q^2-1}-X_1^{q^2-1}}{X_2^{q-1}-X_1^{q-1}}-\frac{X_2^{q^3-1}-X_1^{q^3-1}}{X_2^{q-1}-X_1^{q-1}}+\frac{X_1^{q^2-1}X_2^{q^3-1}-X_1^{q^3-1}X_2^{q^2-1}}{X_2^{q-1}-X_1^{q-1}}.\nonumber
\end{align}
Comparing the homogeneous parts of the two sides of \eqref{eq:C+...} gives
\[
C_{q^2-q}=\frac{X_2^{q^2-1}-X_1^{q^2-1}}{X_2^{q-1}-X_1^{q-1}},
\]
and for $q^2-q<i<q^3-q$, 
\[
C_i-C_{i-(q^2-q)}\frac{X_2^{q^2-1}-X_1^{q^2-1}}{X_2^{q-1}-X_1^{q-1}}+C_{i-(q^2-1)}\frac{X_1^{q-1}X_2^{q^2-1}-X_1^{q^2-1}X_2^{q-1}}{X_2^{q-1}-X_1^{q-1}}=0,
\]
where $C_j$ is defined as 0 for $j<q^2-q$. Using induction on $i$ we see that 
\[
C_{q^2-q}=\frac{X_2^{q^2-1}-X_1^{q^2-1}}{X_2^{q-1}-X_1^{q-1}}\,\Big |\, C_i\quad\text{for all}\ q^2-q\le i<q^3-q.
\]
We also have
\[
C_{q^3-q}=\frac{X_1^{q^2-1}X_2^{q^3-1}-X_1^{q^3-1}X_2^{q^2-1}}{X_1^{q-1}X_2^{q^2-1}-X_1^{q^2-1}X_2^{q-1}}.
\]
By Eisenstein's criterion, $F_3(X_1,X_2,1)$ is absolutely irreducible if we can show that $C_{q^2-q}$ is separable and $\gcd(C_{q^2-q},C_{q^3-q})=1$, equivalently, $C_{q^2-q}(1,X_2)$ is separable and 
\[
\gcd\bigl(C_{q^2-q}(1,X_2),C_{q^3-q}(1,X_2)\bigr)=1.
\]
Since $X_2(X_2^{q^2-1}-1)=X_2^{q^2}-X_2$ is separable, it follows that 
\[
C_{q^2-q}(1,X_2)=\frac{X_2^{q^2-1}-1}{X_2^{q-1}-1}
\]
is separable. We also have
\begin{align*}
&\gcd\bigl(C_{q^2-q}(1,X_2),C_{q^3-q}(1,X_2)\bigr)\cr
=\,&\gcd\Bigl(\frac{X_2^{q^2-1}-1}{X_2^{q-1}-1},\,\frac{X_2^{q^3-1}-X_2^{q^2-1}}{X_2^{q^2-1}-X_2^{q-1}}\Bigr)\cr
\,\Big|\,&\gcd\Bigl(\frac{X_2^{q^2-1}-1}{X_2^{q-1}-1},\,\frac{X_2^{q^3-1}-X_2^{q^2-1}}{X_2^{q-1}-1}\Bigr)\cr
=\,&\frac 1{X_2^{q-1}-1}\gcd(X_2^{q^2-1}-1,\,X_2^{q^3-1}-X_2^{q^2-1})\cr
=\,&\frac 1{X_2^{q-1}-1}\gcd(X_2^{q^2-1}-1,\,X_2^{q^3-1}-1)\cr
=\,&\frac 1{X_2^{q-1}-1}\gcd(X_2^{\gcd(q^2-1,q^3-1)}-1)=1.
\end{align*}
\end{proof}

\begin{lem}\label{L3.2}
For $k\ge 3$, $F_k(X_1,\dots,X_k)$ has an absolutely irreducible factor in $\f_q[X_1,\dots,X_k]$.
\end{lem}

\begin{proof}
We use induction on $k$. When $k=3$, by Lemma~\ref{L3.1}, $F_3(X_1,X_2,X_3)$ is absolutely irreducible.

Now assume $k>3$. It suffices to show that $F_k(X_1,\dots,X_{k-1},1)$ has an absolutely irreducible factor in $\f_q[X_1,\dots,X_k]$. Write
\[
F_k(X_1,\dots,X_{k-1},1)=\frac{\Delta_1(X_1,\dots,X_{k-1},1)}{\Delta(X_1,\dots,X_{k-1},1)}=C_{q^{k-1}-q}+C_{q^{k-1}-q+1}+\cdots,
\]
where $C_i\in\f_q[X_1,\dots,X_{k-1}]$ is homogeneous of degree $i$ and 
\[
C_{q^{k-1}-q}=\frac{\Delta_1(X_1,\dots,X_{k-1})}{\Delta(X_1,\dots,X_{k-1})}=F_{k-1}(X_1,\dots,X_{k-1}).
\]
By the induction hypothesis, $C_{q^{k-1}-q}$ has an absolutely irreducible factor $h\in\f_q[X_1,\dots,X_{k-1}]$. For $1\le i\le k-1$,
\[
\frac{\partial}{\partial X_i}\Delta_1(X_1,\dots,X_{k-1})=\pm\Delta(X_1,\dots,X_{i-1},X_{i+1},\dots,X_{k-1})^{q^2},
\]
which is nonzero and independent of $X_i$. Hence $\Delta_1(X_1,\dots,X_{k-1})$ is separable in each $X_i$, so $\Delta_1(X_1,\dots,X_{k-1})$ is square-free. It follows that $h^2\nmid C_{q^{k-1}-q}$. Now by \cite[Lemma~1.11]{Sze-Thesis-2023}, $F_k(X_1,\dots,X_{k-1},1)$ has an absolutely irreducible factor in $\f_q[X_1,\dots,X_{k-1}]$.
\end{proof}

\noindent{\bf Remark.}
We believe that $F_k(X_1,\dots,X_k)$ ($k\ge 3$) itself is absolutely irreducible. However, having an absolutely irreducible factor in $\f_q[X_1,\dots,X_k]$ is enough for the purpose of the following theorem.

\medskip

In general, for a field $\f$ and a polynomial $F\in\f[X_1,\dots,X_k]$, we define
\[
V_{\f^k}(F)=\{(x_1,\dots,x_k)\in\f^k:F(x_1,\dots,x_k)=0\}.
\]

\begin{thm}\label{T3.3} ($q$-ary version of Theorem~\ref{T1.2}) Assume $3\le k\le n-3$. If
\[
n\ge\frac{13}3k+2\log_q\frac 12(1+\sqrt{21})
\]
or
\[
n\le\frac{13}{10}k-\frac35\log_q\frac 12(1+\sqrt{21}),
\]
$g_{q-1}$ is not $k$th order sum-free on $\f_{q^n}$.
\end{thm}

\begin{proof}
By the statement before Lemma~\ref{L3.1}, to prove that $g_{q-1}$ is not $k$th order sum-free on $\f_{q^n}$, it suffices to show that 
\[
\bigl|V_{\f_{q^n}^k}(F_k)\setminus V_{\f_{q^n}^k}(\Delta)\bigr|>0.
\]
Since $F_k$ has an absolutely irreducible factor $h\in\f_q[X_1,\dots,X_k]$ (Lemma~\ref{L3.2}), by the Lang-Weil bound as stated in \cite[Theorem~5.2]{Cafure-Matera-FFA-2006}, we have
\begin{align}\label{eq:LW}
&|V_{\f_{q^n}^k(}(F_k)|\ge|V_{\f_{q^n}^k}(h)|\\
&\ge q^{n(k-1)}-(q^k-q-1)(q^k-q-2)q^{n(k-3/2)}-5(q^k-q)^{13/3}q^{n(k-2)}\cr
&>q^{n(k-1)}-q^{2k}q^{n(k-3/2)}-5q^{13k/3}q^{n(k-2)}\cr
&=q^{n(k-2)}\bigl[q^n-q^{2k}q^{n/2}-5q^{13k/3}\bigr].\nonumber
\end{align}
Since $\Delta_1(X_1,\dots,X_k)$ is square-free, we have $\gcd(F_k,\Delta)=1$. Hence by \cite[Lemma~2.2]{Cafure-Matera-FFA-2006}
\begin{equation}\label{eq:Bezout}
\bigl|V_{\f_{q^n}^k}(F_k)\cap V_{\f_{q^n}^k}(\Delta)\bigr|\le\max\{\deg F_k,\,\deg\Delta\}^2q^{n(k-2)}\le q^{2k}\cdot q^{n(k-2)}.
\end{equation}
Combining \eqref{eq:LW} and \eqref{eq:Bezout} gives
\[
\bigl|V_{\f_{q^n}^k}(F_k)\setminus V_{\f_{q^n}^k}(\Delta)\bigr|>q^{n(k-2)}\bigl[q^n-q^{2k}q^{n/2}-5q^{13k/3}-q^{2k}\bigr].
\]
Let $x_0$ denote the larger root of the quadratic $X^2-q^{2k}X-5q^{13k/3}-q^{2k}\in\Bbb R[X]$.Then
\begin{align*}
x_0\,&=\frac 12\bigl[q^{2k}+(q^{4k}+20q^{13k/3}+4q^{2k})^{1/2}\bigr]\cr
&\le\frac 12\bigl[q^{2k}+(21q^{13k/3})^{1/2}\bigr]\cr
&\le\frac 12(1+\sqrt{21})q^{13k/6}.
\end{align*}
When $q^{n/2}\ge x_0$, i.e., when
\[
n\ge\frac{13}3k+2\log_q\frac 12(1+\sqrt{21}),
\]
we have 
\[
\bigl|V_{\f_{q^n}^k}(F_k)\setminus V_{\f_{q^n}^k}(\Delta)\bigr|>0,
\]
whence $g_{q-1}$ is not $k$th order sum-free on $\f_{q^n}$.

Since $g_{q-1}$ is $k$th order sum-free on $\f_{q^n}$ if and only if it is $(n-k)$-th order sum-free on $\f_{q^n}$, when 
\[
n\ge\frac{13}3(n-k)+2\log_q\frac 12(1+\sqrt{21}),
\]
i.e., when
\[
n\le\frac{13}{10}k-\frac35\log_q\frac 12(1+\sqrt{21}),
\]
$g_{q-1}$ is not $k$th order sum-free on $\f_{q^n}$ neither.
\end{proof}


\section{A Characterization of Sum-Free Functions}

Denote the elements of $\f_{2^n}$ by $x_0,\dots,x_{2^n-1}$ and identify them with the column vectors in $\f_2^n$. For a function $f:\f_{2^n}\to\f_{2^n}$, let $C(f)$ denote the binary linear code of length $2^n$ with parity check matrix
\begin{equation}\label{eq:mtx}
\left[\begin{matrix}
1&1&\cdots&1\cr
x_0&x_1&\cdots&x_{2^n-1}\cr
f(x_0)&f(x_1)&\cdots&f(x_{2^n-1})\end{matrix}\right].
\end{equation}
APN functions can be characterized in terms of the code $C(f)$.

\begin{thm}\label{t4.1}\cite{Carlet-2021, Carlet-Charpin-Zinoviev-DCC-1998}
For $n\ge 4$, a function $f:\f_{2^n}\to \f_{2^n}$ is APN if and only if $C(f)$ is a $[2^n,2^n-2n-1,6]$ code.
\end{thm}

\begin{rmk}\label{r4.2}\rm
Theorem~\ref{t4.1} was proved, in an equivalent form, by Carlet, Charpin and Zinoviev \cite{Carlet-Charpin-Zinoviev-DCC-1998}, and was stated in the present form in \cite{Carlet-2021}. It is not difficult to see that $f$ is APN if and only if $C(f)$ is a $[2^n,k,d]$ code with $k\ge 2^n-2n-1$ and $d\ge 6$. The sphere packing bound shows that $d\le 6$, and by a result of \cite{Brouwer-Tolhuizen-DCC-1993}, we have $k\le 2^n-2n-1$.

When $n=3$, $f:\f_{2^3}\to\f_{2^3}$ is APN if and only if $C(f)$ is an $[8,1,8]$ code.
\end{rmk}

We will see that the characterization in Theorem~\ref{t4.1} can be generalized to binary and $q$-ary sum-free functions of any order.

\subsection{The Reed-Muller code}\

A polynomial $f\in\f_q[X_1,\dots,X_n]$ is said to be {\em reduced} if the degree of $f$ in $X_i$ is at most $q-1$ for all $1\le i\le n$. Every function $g:\f_q^n\to\f_q$ is uniquely represented by a reduced polynomial $f\in\f_q[X_1,\dots,X_n]$. We identify $g$ with $f$ and define $\deg g=\deg f$. For $-1\le r\le n(q-1)$, the $r$th order {\em Reed-Muller} code of length $q^n$ is
\begin{equation}\label{eq:rm} 
R_q(r,n)=\{f:\f_q^n\to\f_q:\deg f\le r\}.
\end{equation}
When $q=2$, we write $R_2(r,n)=R(r,n)$. Note that $R_q(n(q-1),n)$ is the $\f_q$-algebra of all functions from $\f_q^n$ to $\f_q$ and $R_q(-1,n)=\{0\}$. Each function $\f_q^n\to \f_q$ is identified with a vector $(f(x_0),\dots,f(x_{q^n-1}))$, where $\{x_0,\dots,x_{q^n-1}\}=\f_q^n$. The Hamming weight of $f$ is $|f|=|\{x\in\f_q^n:f(x)\ne 0\}|$. Therefore, $R_q(r,n)$ is a linear code of length $q^n$ over $\f_q$; its other parameters are
\begin{equation}\label{eq:dim}
\text{dimension}=\sum_{i\le\lfloor r/q\rfloor}(-1)^i\binom ni\binom{r-qi+n}n,
\end{equation}
\begin{equation}\label{eq:minwt}
\text{min weight}=(q-t)q^{n-s-1},\quad\text{where}\ r=s(q-1)+t,\ 0\le t<q-1,
\end{equation}
see \cite[\S 2.3 and Corollary~5.12]{Hou-ams-gsm-2018}. The binary Reed-Muller code $R(r,n)$ ($0\le r\le n$) has dimension $\sum_{i=0}^r\binom ni$ and minimum weight $2^{n-r}$. The dual of $R_q(r,n)$ is $R_q(r,n)^\bot= R_q(r',n)$, where $r+r'=n(q-1)-1$.

The {\em next-to-minimum weight} of a code is the second smallest nonzero weight of the code. The next-to-minimum weight of the Reed-Muller code was determined in \cite{Kasami-Tokura-IEEE-IT-1970} for $q=2$ and in \cite{Bruen-CM523-2010, Erickson-thesis-1974} for general $q$.

\begin{thm}\label{t4.3}\cite[Theorem~1]{Kasami-Tokura-IEEE-IT-1970} For $0<r\le n$, the next-to-minimum weight of $R(r,n)$ is
\begin{equation}\label{eq:2wt}
2^{n-r}+c 2^{n-r-1},
\end{equation}
where $2^{n-r}$ is the minimum weight of $R(r,n)$ and
\[
c=\begin{cases}
2&\text{if}\ r=1,\cr
1&\text{if}\ 2\le r\le n-2,\cr
2&\text{if}\ r=n-1,n.
\end{cases}
\]
\end{thm}

\begin{thm}\label{t4.4}\cite[Theorem~4.5]{Bruen-CM523-2010} \cite[Theorem~3.1]{Erickson-thesis-1974}
Assume $q>2$, and let $0<r\le n(q-1)$, written in the form $r=s(q-1)+t$, $0<t\le q-1$. Then the next-to-minimum weight of $R_q(r,n)$ is 
\begin{equation}\label{eq:2wtq}
(q-t)q^{n-s-1}+cq^{n-s-2},
\end{equation}
where $(q-t)q^{n-s-1}$ is the minimum weight of $R_q(r,n)$ and
\[
c=\begin{cases}
q&\text{if}\ s=n-1,\cr
t-1&\text{if}\ s\le n-2,\ t>1\cr
q&\text{if}\ 0<s\le n-2,\ q\ge 4,\ t=1,\cr
q-1&\text{if}\ 0<s\le n-2,\ q=3,\ t=1,\cr
q&\text{if}\ s=0,\ t=1.
\end{cases}
\]
\end{thm}

Note that in Theorem~\ref{t4.4}, $0<t\le q-1$, while in Equation~\eqref{eq:minwt}, $0\le t<q-1$. Also note that when $r=0$, $R_q(0,n)$ does not have a next-to-minimum weight.

The minimum weight codewords of $R_q(r,n)$ were determined by the Delsarte-Goethals-MacWilliams theorem \cite{Delsarte-Goethals-MacWilliams-IC-1970}; also see \cite[Theorem~5.15]{Hou-ams-gsm-2018}. In particular, when $r=s(q-1)$, the minimum weight codewords of $R_q(r,n)$ are scalar multiples of indicator functions of $(n-s)$-dimensional $\f_q$-affine subspaces of $\f_q^n$. The indicator function of a subset $A\subset\f_q^n$ is defined as
\[
\begin{array}{cccl}
i_A:&\f_q^n&\longrightarrow&\f_q\vspace{0.3em}\cr
&x&\longmapsto &\begin{cases}
1&\text{if}\ x\in A,\cr
0&\text{if}\ x\notin A.
\end{cases}
\end{array}
\]

\subsection{A coding theoretic characterization of sum-free functions}\

Denote the elements of $\f_{q^n}$ by $x_0,\dots,x_{q^n-1}$ and identify them with the column vectors in $\f_q^n$. For $1\le s\le n-1$, let $A_s$ be a generator matrix of $R_q(s(q-1)-1,n)$ whose columns are indexed by $x_0,\dots,x_{q^n-1}$. For any function $f:\f_{q^n}\to\f_{q^n}$, let $C_s(f)$ denote the linear code over $\f_q$ with parity check matrix
\begin{equation}\label{eq:mtx-q}
\left[
\begin{matrix}
&A_s\cr
f(x_0)&\cdots&f(x_{q^n-1})\end{matrix}
\right]
\end{equation}
When $q=2$ and $s=2$, this is the matrix in \eqref{eq:mtx}.

\begin{thm}\label{t4.5}
Assume $q=2$ and $1\le s\le n-1$. A function $f:\f_{2^n}\to \f_{2^n}$ is $s$th order sum-free if and only if $C_s(f)$ is a $[2^n,k,d]$ code over $\f_2$ with
\begin{equation}\label{eq:k-ge}
k\ge 2^n-n-\sum_{i=0}^{s-1}\binom ni,
\end{equation}
\begin{equation}\label{eq:d-ge}
d\ge\begin{cases}
4&\text{if}\ s=1,\cr
2^s+2^{s-1}&\text{if}\ 2\le s\le n-2,\cr
2^n&\text{if}\ s=n-1.
\end{cases}
\end{equation}
\end{thm}

\begin{proof}
\eqref{eq:k-ge} is obvious. Since $R(s-1,n)^\bot=R(n-s,n)$, we have
\[
C_s(f)=\{u\in R(n-s,n): (f(x_0),\dots,f(x_{q^n-1}))u^T=0\}.
\]
The right side of $\eqref{eq:d-ge}$ is the next-to-minimum weight of $R(n-s,n)$. Therefore,
\begin{align*}
&\text{\eqref{eq:d-ge} holds}\cr
\Leftrightarrow\ &\text{for each minimum weight codeword $u\in R(n-s,n)$},\cr
&(f(x_0),\dots,f(x_{2^n-1}))u^T\ne 0,\cr
\Leftrightarrow\ & \text{for each indicator vector $u$ of an $s$-dimensional $\f_2$-affine subspace of $\f_{2^n}$},\cr
&(f(x_0),\dots,f(x_{2^n-1}))u^T\ne 0,\cr
\Leftrightarrow\ & \sum_{x\in A}f(x)\ne 0\ \text{for every $s$-dimensional $\f_2$-affine subspace $A$ of $\f_{2^n}$},\cr
\Leftrightarrow\ &\text{$f$ is $s$th order sum-free}.
\end{align*}
\end{proof}

\begin{thm}\label{t4.6}
Assume $q>2$ and $1\le s\le n-1$. A function $f:\f_{q^n}\to\f_{q^n}$ is $s$th order sum-free if and only if $C_s(f)$ is a $[q^n,k,d]$ code over $\f_q$ with
\begin{equation}\label{eq:qk-ge}
k\ge q^n-n-\dim_{\f_q}R_q(s(q-1)-1,n),
\end{equation}
\begin{equation}\label{eq:qd-ge}
d\ge q^s+(q-2)q^{s-1}.
\end{equation}
\end{thm}

\begin{proof}
The proof is identical to that of Theorem~\ref{t4.5}. We only have to note that the right side of \eqref{eq:qd-ge} is the next-to-minimum weight of $R_q(s(q-1)-1,n)^\bot=R_q((n-s)(q-1),n)$ and that the minimum weight codewords of $R_q((n-s)(q-1),n)$ are indicator vectors of $s$-dimensional $\f_q$-affine subspaces of $\f_{q^n}$.
\end{proof}

\begin{rmk}\label{R4.7}\rm
As we mentioned in the introduction, Theorem~\ref{t4.5} was covered by the main result of \cite{Heering-Kaspers-Taranchuk-2026}. When $f:\f_{2^n}\to\f_{2^n}$ is $s$th order sum-free with $2\le s\le n-2$, it was also proved in \cite{Heering-Kaspers-Taranchuk-2026} that the minimum weight of $C_s(f)$ is $3\cdot 2^{s-1}$. Theorem~\ref{t4.6} is new to our best knowledge. 
\end{rmk}


\section{A Study of the Welch Function}\label{sec5}

Recall that the Welch function on $\f_{2^n}$, where $n=2m+1$, is defined by $W_n(X)=X^{2^m+3}$. It is known that $W_n(X)$ is APN, i.e., 2nd order sum-free \cite{Dobbertin-IEEE-IT-1999}. In this section, we will give an alternative proof of this result and we will also investigate the 3rd order sum-freedom of $W_n(X)$.

\subsection{An alternative proof for the 2nd order sum-freedom of $W_n(X)$}\

The proof given here is based on the fact that with computer assistance, the computation and factorization of resultants of polynomials of moderate degree in few variables are effortless; such tasks were laborious in the past.

Assume to the contrary that $W_n(X)$ is not 2nd order sum-free, that is, there is a 2-dimensional $\f_2$-affine subspace $A$ of $\f_{2^n}$ such that $\sum_{a\in A}W_n(a)=0$. We may assume, without loss of generality, that $A=z+\langle 1,x\rangle$, where $x\in\f_{2^n}\setminus\f_2$. We have
\begin{align}\label{eq:sumA=0}
0\,&=\sum_{a\in z+\langle 1,x\rangle}a^{2^m+2+1}=\sum_{a_1,a_2\in\f_2}(z+a_1+a_2x)^{2^m+2+1}\\
&=\sum_{a_1,a_2\in\f_2}(z^{2^m}+a_1+a_2x^{2^m})(z^2+a_1+a_2x^2)(z+a_1+a_2x)\cr
&=x^{2^m+2}+x^{2^m+1}+x^3+x^{2^m}+x^2+x+z^{2^m}(x+x^2)+z^2(x+x^{2^m})+z(x^2+x^{2^m})\cr
&=(x+x^{2^m})(1+x+x^2)+(z^{2^m}+z)(x+x^2)+(x+x^{2^m})(z^2+z)\nonumber
\end{align}
Let 
\[
s=x+x^{2^m},\quad t=s^{2^m},\quad u=z+z^{2^m},\quad v=u^{2^m}.
\]
Then \eqref{eq:sumA=0} becomes
\begin{equation}\label{eq:F=0}
F(s,t,u,v)=0,
\end{equation}
where
\begin{equation}\label{eq:F(stuv)}
F(S,T,U,V)=S(1+S^2+T^2)+U(S^2+T^2)+S(U^2+V^2).
\end{equation}
Since $x\notin\f_2$, we have $s\ne 0$. Since $\tr_{2^n/2}(1+u+v)=1$, we have $1+u+v\ne 0$. Note that
\[
(\ )^{2^{m+1}}:\quad\begin{cases}s\mapsto t^2,\cr t\mapsto s,\end{cases}\quad
\begin{cases}u\mapsto v^2,\cr v\mapsto u.\end{cases}
\]
Hence for each $g\in\f_2[S,T,U,V]$,
\[
g(s,t,u,v)^{2^{m+1}}=g(t^2,s,v^2,u).
\]
In particular, 
\begin{equation}\label{eq:imply}
g(s,t,u,v)=0\quad \text{implies}\quad g(t^2,s,v^2,u)=0.
\end{equation}
This fact, together with \eqref{eq:F=0} and the fact $s(1+u+v)\ne 0$, is all we need to derive a contradiction.

First, we have
\[
0=\res\bigl(F(s,t,u,v),\,F(t^2,s,v^2,u);\, u\bigr)=f_1(s,t,v)^2f_2(s,t,v)^2,
\]
where $\res(\,\cdot\,,\,\cdot\,\;u)$ denotes the resultant of two polynomials with respect to $u$, and
\begin{align*}
&f_1(s,t,v)=t^2+t^2s+tsv+sv^2,\cr
&f_2(s,t,v)=t^2+t^4+t^2s+s^2+t^3v+t^2v^2.
\end{align*}
Hence $f_1(s,t,v)=0$ or $f_2(s,t,v)=0$.

We claim that $f_2(s,t,v)=0$. Otherwise, $f_1(s,t,v)=0$. Then 
\[
0=\res(f_1(s,t,v), F(t^2,s,v^2,u);s)=t^2f_3(t,u,v)^2,
\]
where
\[
f_3(t,u,v)=t^4+t^2u+tuv+v^2+t^2v^2+uv^2+v^4.
\]
We have
\begin{align*}
0\,&=\res(f_1(s,t,v),f_3(t,u,v);v)=t^4(t^4+s^2+t^2s^2+s^4+s^3u+s^2u^2)\cr
&=t^4f_2(t^2,s,u).
\end{align*}
Hence $f_2(t^2,s,u)=0$, which implies $0=f_2(s^2,t^2,v^2)=f_2(s,t,v)^2$, which is a contradiction. Hence the claim is proved.

Now
\[
0=\res(f_2(t^2,s,u),F(s,t,u,v);s)=t^4f_3(t,u,v)^2.
\]
Hence $f_3(t,u,v)=0$. On the other hand,
\[
0=\res(f_2(s,t,v),F(t^2,s,v^2,u);s)=t^4f_4(t,u,v)^2,
\]
where
\[
f_4(t,u,v)=1+t^4+t^2u+u^2+tv+t^3v+tuv+v^2+t^2v^2.
\]
We have
\[
0=\res(f_3(t,u,v),f_4(t,u,v);t)=(1+u+v)^4(1+u+v^2)^4.
\]
By \eqref{eq:imply}, $1+u+v^2=0$ implies $1+u+v=0$. Hence we always have $1+u+v=0$, which is a contradiction. This completes the proof.

\begin{rmk}\label{R4.1}\rm
The above proof is essentially based on the same idea as the proof in \cite{Dobbertin-IC-1999}. What differs the two proofs is in their approaches to polynomial equations over finite fields. The proof of \cite{Dobbertin-IC-1999} relied on clever trace arguments. Our strategy is to reduce polynomial equations continuously using resultants, hiding the subtleties behind computations. Our approach can be placed in the framework of a more general algebraic question. 
\end{rmk} 

\begin{ques}\label{Q4.2}\rm
Let $\f$ be a field and $I$ an ideal of $\f[X_1.\dots,X_k]$ with the following properties:
\begin{itemize}
\item[(i)] $I$ contains one or several given polynomials.

\item[(ii)] If $g,h\in\f[X_1,\dots,X_k]$ are such that $gh\in I$, then either $g\in I$ or $h\in I$. ($I$ is not necessarily a prime ideal because it could be the entire ring $\f[X_1,\dots,X_k]$.)

\item[(iii)] There exists $\sigma=(\sigma_1,\dots,\sigma_k)\in\f[X_1,\dots,X_k]^k$ such that $g\in I$ implies $g\circ\sigma\in I$.
\end{itemize}
The objective is to find a few elements in $\f[X_1,\dots,X_k]$, with degrees as low as possible, such that $I$ contains at least one of them.
\end{ques}

In the above proof, $I$ is the ideal of $\f_2[S,T,U,V]$ defined by
\[
I=\{g\in\f_2[S,T,U,V]:g(s,t,u,v)=0\},
\]
$I$ contains $F$, and $\sigma=(T^2,S,V^2,U)$. We showed that $I$ must contain $S$ or $1+U+V$.

\subsection{Dickson matrices}\

The investigation of the 3rd order sum-freedom of the Welch function involves Dickson matrices. Here we collect some general facts about these matrices.

Let $q$ be any prime power. For $a_0,\dots,a_{n-1}\in\f_{q^n}$, the matrix
\[
\mathcal D(a_0,\dots,a_{n-1})=\left[
\begin{matrix}
a_0&a_1&\cdots&a_{n-1}\cr
a_{n-1}^q&a_0^q&\cdots&a_{n-2}^q\cr
\vdots&\vdots&\ddots&\vdots\cr
a_1^{q^{n-1}}&a_2^{q^{n-1}}&\cdots&a_0^{q^{n-1}}
\end{matrix}\right]
\]
is called a {\em Dickson matrix}.

\begin{thm}\label{T4.3}\cite[Theorem~2.29]{Hou-ams-gsm-2018}
For $a_0,\dots,a_{n-1}\in\f_{q^n}$, the roots of the $q$-polynomial $a_0X+a_1X^q+\cdots+a_{n-1}X^{q^{n-1}}$ in $\f_{q^n}$ is a vector space over $\f_2$ of dimension $n-\rank\mathcal D(a_0,\dots,a_{n-1})$.
\end{thm}

\begin{defn}\label{D4.4} For an $m\times n$ matrix $A$, a set of rows with indices $i,i+1,\dots,i+k \pmod m$ is called a set of {\em continuous rows}. A set of continuous columns is defined the same way. A {\em continuous submatrix} of $A$ is a submatrix whose row indices are $i,i+1,\dots,i+k \pmod m$ and whose column indices are $j,j+1,\dots,j+l \pmod n$.
\end{defn}

\begin{lem}\label{L4.5}
Let $A$ be an $n\times n$ Dickson matrix over $\f_{q^n}$ with $\rank A=r$. Then any set of $r$ continuous rows (columns) are linearly independent over $\f_{q^n}$.
\end{lem}

\begin{proof}
Write
\[
A=\left[\begin{matrix} w_1\cr\vdots\cr w_n\end{matrix}\right].
\]
Assume to the contrary and without loss of generality that $w_1,\dots,w_r$ are linearly dependent over $\f_{q^n}$. Then there exists $s\le r$ such that $w_s$ is a linear combination of $w_1,\dots,w_{s-1}$, say
\[
w_s=b_1w_1+\cdots+b_{s-1}w_{s-1},\quad b_i\in\f_{q^n}.
\]
Apply $(\ )^q$ to the components of both sides of the above equation. We have
\begin{align*}
w_{s+1}\,&=b_1^qw_2+\cdots+b_{s-2}^qw_{s-1}+b_{s-1}^qw_s\cr
&=b_{s-1}^qb_1w_1+(b_{s-1}^qb_2+b_1^q)w_2+\cdots+(b_{s-1}^qb_{s-1}+b_{s-2}^q)w_{s-1}\cr
&\in\langle w_1,\dots,w_{s-1}\rangle.
\end{align*}
Continuing this way, we see that $w_i\in\langle w_1,\dots,w_{s-1}\rangle$ for all $s\le i\le n$. Then $\rank A\le s-1$, which is a contradiction.
\end{proof}

\begin{cor}\label{C4.6}
Let $A$ be an $n\times n$ Dickson matrix over $\f_{q^n}$ with $\rank A=r$. Then every $r\times r$ continuous submatrix of $A$ is nonsingular.
\end{cor}

\begin{proof}
Without loss of generality, we only prove that the $r\times r$ leading principal submatrix of $A$ is nonsingular. Write
\[
A=[c_1,\dots,c_n].
\]
By Lemma~\ref{L4.5}, $\rank[c_1,\dots,c_r]=r$, and the first $r$ rows of $[c_1,\dots,c_r]$ span the row space of $[c_1,\dots,c_r]$. Therefore, the $r\times r$ leading principal submatrix of $A$, which consists of the first $r$ rows of $[c_1,\dots,c_r]$, is nonsingular.
\end{proof}

\begin{prop}\label{P4.7}
Let $A$ be an $n\times n$ Dickson matrix over $\f_{q^n}$. Then the following statements are equivalent:
\begin{itemize}
\item[(i)] $\rank A=r$.

\item[(ii)] All $r\times r$ continuous submatrices of $A$ are nonsingular, and for each $r<i\le n$, all $i\times i$ submatrices of $A$ are singular.

\item[(iii)] $A$ has an $r\times r$ nonsingular submatrix, and for each $r<i\le n$, $A$ has an $i\times i$ singular continuous submatrix.
\end{itemize}
\end{prop}

\begin{proof}
(i) $\Rightarrow$ (ii). This is Corollary~\ref{C4.6}.

\medskip
(ii) $\Rightarrow$ (i) Obvious.

\medskip
(ii) $\Rightarrow$ (iii). Obvious.

\medskip
(iii) $\Rightarrow$ (i). By Corollary~\ref{C4.6}, $\rank A\ne i$ for all $r<i\le n$. On the other hand, we clearly have $\rank A\ge r$.
\end{proof}

\subsection{Third order sum-freedom of the Welch function}\label{sec5.3}\

Recall that $n=2m+1$ and the Welch function on $\f_{2^n}$ is $W_n(X)=X^{2^m+3}$. When $m=1$, the algebraic degree of $W_3$ is 2, hence by Remark~\ref{R2.4}, $W_3$ is not 3rd order sum-free. When $m\ge 2$, the algebraic degree of $W_n$ is 3, hence by Remark~\ref{R2.4}, $W_n$ is not $k$th order sum-free for $4\le k\le n$. When $m=2$, by Theorem~\ref{T1.1}, $f$ is 3rd order sum-free. For an arbitrary $m\ge 3$, the 3rd order sum-freedom of $W_n$ is not known.

Assume $m\ge 2$ and let $A$ be an arbitrary 3-dimensional $\f_2$-affine subspace of $\f_{2^n}$. The question is to determine if the sum $\sum_{a\in A}f(a)$ can be 0. By Lemma~\ref{L2.1}, we may assume that $A$ is an $\f_2$-linear subspace of $\f_{2^n}$. Since $W_n(X)$ is a power function, we may further assume that $1\in A$, say $A$ has a basis $1,x,y$ over $\f_2$. Then we have 
\begin{align*}
\sum_{a\in A}f(a)\,&=\sum_{a_1,a_2,a_3\in\f_2}(a_1x+a_2y+a_3)^{2^m+2+1}\cr
&=\sum_{a_1,a_2,a_3\in\f_2}(a_1x^{2^m}+a_2y^{2^m}+a_3)(a_1x^2+a_2y^2+a_3)(a_1x+a_2y+a_3)\cr
&=\sum_{a_1,a_2,a_3\in\f_2}a_1a_2a_3(x^{2^m}y^2+x^{2^m}y+x^2y^{2^m}+x^2y+xy^{2^m}+xy^2)\cr
&=x^{2^m}y^2+x^{2^m}y+x^2y^{2^m}+x^2y+xy^{2^m}+xy^2\cr
&=g(x,y),
\end{align*}
where
\[
g(X,Y)=Y^{2^m}(X^2+X)+Y^2(X^{2^m}+X)+Y(X^{2^m}+X^2).
\]
Therefore, $W_n(X)$ is not 3rd order sum-free if and only if 
\begin{equation}\label{eq:g=0}
g(x,y)=0\quad\text{for some}\ x\in\f_{2^n}\setminus\f_2\ \text{and}\ y\in\f_{2^n}\setminus\langle 1,x\rangle.
\end{equation}
Clearly, $g(x,y)=0$ when $y\in\langle 1,x\rangle$. Note that $g(x,Y)$ is a 2-polynomial in $Y$. By Theorem~\ref{T4.3}, \eqref{eq:g=0} holds if and only if the Dickson matrix
\begin{equation}\label{eq:D(x)}
D(x):=\mathcal D(\underset{\rule{0pt}{1em}0}{x^{2^m}+x^2},\,\underset{\rule{0pt}{1em}1}{x^{2^m}+x},\,0,\cdots,0,\,\underset{\rule{0pt}{1em}m}{x^2+x},\,0,\cdots,\underset{\rule{0pt}{1em}n-1}0)
\end{equation}
has rank $\le n-3$ for some $x\in\f_{2^n}\setminus\f_2$. Clearly, $\rank D(x)\le n-2$ (since $g(x,y)=0$ for $y\in\langle 1,x\rangle$. Thus by Proposition~\ref{P4.7}, $\rank D(x)\le n-3$ if and only if some (and hence every) $(n-2)\times(n-2)$ continuous submatrix of $D(x)$ is singular. Therefore, we have the following theorem.

\begin{thm}\label{T4.8}
Let $n\ge 5$ and let $D^*(x)$ denote the leading $(n-2)\times(n-2)$ principal submatrix of the Dickson matrix $D(x)$ in \eqref{eq:D(x)}. Then the Welch function $W_n(X)$ on $\f_{2^n}$ is not 3rd order sum-free if and only if $\det D^*(x)=0$ for some $x\in\f_{2^n}\setminus\f_2$.
\end{thm}

We may replace the element $x$ in \eqref{eq:D(x)} with an indeterminate $X$; the resulting matrix is $D(X)$, and its leading $(n-2)\times(n-2)$ principal submatrix is $D^*(X)$.  The roots of $\det D^*(X)$ in $\f_{2^n}$ are precisely the roots of $\gcd(\det D^*(X),X^{2^n}+X)$ and $0,1$ are among these roots. Therefore, we have the following corollary.

\begin{cor}\label{C4.9}
The Welch function $W_n(X)$ on $\f_{2^n}$ ($n\ge 5$) is not 3rd order sum-free if and only if 
\begin{equation}\label{eq:deggcd}
\deg\gcd(\det D^*(X),\,X^{2^n}+X)>2.
\end{equation}
\end{cor}

We computed the above degree for $5\le n=2m+1\le 15$ (see Table~\ref{Tb1}), and we propose the following conjecture.

\begin{conj}\label{C4.10}
For $n=2m+1\ge 7$, the Welch function on $\f_{2^n}$ is not $3$rd order sum-free.
\end{conj}

\begin{table}
\caption{$\deg\gcd(\det D^*(X),\,X^{2^n}+X)$}\label{Tb1}
\renewcommand*{\arraystretch}{1.2}
\centering
\begin{tabular}{c|c}
$n=2m+1$ & $\deg\gcd(\det D^*(X),\,X^{2^n}+X)$\\ \hline
5&2 \\
7&44 \\
9&134 \\
11&464 \\
13&2186 \\
15&7988\\
\end{tabular}
\end{table}

Table~\ref{Tb1} shows that the conjecture is true for $7\le n\le 15$. Moreover, it suggests that the degree in \eqref{eq:deggcd} grows rapidly as $n$ increases.

\subsection{Power functions of algebraic degree 3}\

Consider 
\[
f_{n,i_1,i_2}(X)=X^{2^0+2^{i_1}+2^{i_2}},
\]
as a function from $\f_{2^n}$ to $\f_{2^n}$, where $0<i_1<i_2<n$. The computations in \S\ref{sec5.3} on the Welch function carry over to the function $f_{n,i_1,i_2}$. Let $A_{n,i_1,i_2}(X)$ denote the leading $(n-2)\times(n-2)$ principal submatrix of the Dickson matrix
\[
\mathcal D(\underset{\rule{0pt}{1em}0}{X^{2^{i_1}}+X^{2^{i_2}}},\,\underset{\rule{0pt}{1em}i_1}{X^{2^0}+X^{2^{i_2}}},\,0,\cdots,0,\,\underset{\rule{0pt}{1em}i_2}{X^{2^0}+X^{2^{i_1}}},\,0,\cdots,\underset{\rule{0pt}{1em}n-1}0),
\]
and let
\[
d_{n,i_1,i_2}=\deg\gcd(\det A_{n,i_1,i_2}(X),\,X^{2^n}+X).
\]
Then $f_{n,i_1,i_2}$ is 3rd order sum-free if and only if $d_{n,i_1,i_2}=2$. We computed $d_{n,i_1,i_2}$ for $0<i_1<i_2<n\le 13$, and we found that $d_{n,i_1,i_2}=2$ precisely when $\{0,i_1,i_2\}\overset n\sim\{0,1,2\}$, where the equivalence $\overset n\sim$ is defined at the end of Section~\ref{sec2}. This leads to the following conjecture.

\begin{conj}\label{C5.11}
For $0<i_1<i_2<n$, the function $f_{n,i_1,i_2}(X)$ is $3$rd order sum-free on $\f_{2^n}$ if and only if $\{0,i_1,i_2\}\overset n\sim\{0,1,2\}$.
\end{conj}

It is easy to show that when $n=2m+1\ge 7$, $\{0,1,m\}\overset n{\not\sim}\{0,1,2\}$. Hence Conjecture~\ref{C5.11} implies Conjecture~\ref{C4.10}.

In fact, a more general form of Conjecture~\ref{C5.11} for $k$th order sum-free functions of algebraic degree $k$ was stated informally in \cite[\S5.3]{Carlet-ePrint-2024}. Because our computational results only concern the case $k=3$, we choose to focus on this case for the time being.

\section*{Acknowledgment}

Xiang-dong Hou was partially supported by NSF RTG grant 2342254.



\end{document}